\documentclass[a4, 11pt] {amsart} 
\usepackage{mathrsfs}
\usepackage{amssymb,times, amscd,amsmath,amsthm, xypic}
\usepackage[all]{xy}
\usepackage[dvips]{graphicx, color}
\usepackage{amsmath}
\usepackage{amssymb}
\usepackage[all]{xy}

\usepackage{graphicx}

\begin{document}

\numberwithin{equation}{section}
\newtheorem{thm}[equation]{Theorem}
\newtheorem{pro}[equation]{Proposition}
\newtheorem{prob}[equation]{Problem}
\newtheorem{qu}[equation]{Question}
\newtheorem{cor}[equation]{Corollary}
\newtheorem{con}[equation]{Conjecture}
\newtheorem{lem}[equation]{Lemma}
\theoremstyle{definition}
\newtheorem{ex}[equation]{Example}
\newtheorem{defn}[equation]{Definition}
\newtheorem{ob}[equation]{Observation}
\newtheorem{rem}[equation]{Remark}

\hyphenation{homeo-morphism} 

\newcommand{\calA}{\mathcal{A}} 
\newcommand{\calD}{\mathcal{D}} 
\newcommand{\calE}{\mathcal{E}}
\newcommand{\calC}{\mathcal{C}} 
\newcommand{\Set}{\mathcal{S}et\,} 
\newcommand{\Top}{\mathcal{T}\!op \,}
\newcommand{\Topst}{\mathcal{T}\!op\, ^*}
\newcommand{\calK}{\mathcal{K}} 
\newcommand{\calO}{\mathcal{O}} 
\newcommand{\calS}{\mathcal{S}} 
\newcommand{\calT}{\mathcal{T}} 
\newcommand{\Z}{{\mathbb Z}}
\newcommand{\C}{{\mathbb C}}
\newcommand{\Q}{{\mathbb Q}}
\newcommand{\R}{{\mathbb R}}
\newcommand{\N}{{\mathbb N}}
\newcommand{\F}{{\mathcal F}} 

\def\op{\operatorname}

\hfill

\title{Poincar\'e polynomials of a map and \\
a relative Hilali conjecture}  

\author{ Toshihiro Yamaguchi \ \ and \ \ Shoji Yokura}
\date{}

\address{Faculty of Education, Kochi University, 2-5-1,Kochi, 780-8520, Japan} 
\email{tyamag@kochi-u.ac.jp}

\address{Department of Mathematics  and Computer Science, Graduate School of Science and Engineering, Kagoshima University, 1-21-35 Korimoto, Kagoshima, 890-0065, Japan
}

\email{yokura@sci.kagoshima-u.ac.jp}

\maketitle 

\begin{abstract}  In this paper we introduce homological and homotopical Poincar\'e polynomials $P_f(t)$ and $P^{\pi}_f(t)$ of a continuous map $f:X \to Y$ such that if $f:X \to Y$ is a constant map, or more generally, if $Y$ is contractible, then these Poincar\'e polynomials are respectively equal to the usual homological and homotopical Poincar\'e polynomials $P_X(t)$ and $P^{\pi}_X(t)$ of the source space $X$. Our relative Hilali conjecture $P^{\pi}_f(1) \leqq P_f(1)$ is a map version of the  the well-known Hilali conjecture $P^{\pi}_X(1) \leqq P_X(1)$ of a rationally elliptic space X. In this paper we show that under the condition 
that $H_i(f;\mathbb Q):H_i(X;\mathbb Q) \to H_i(Y;\mathbb Q)$ is not injective for some $i>0$, the relative Hilali conjecture of product of maps holds, namely, there exists a positive integer $n_0$ such that for $\forall n \geqq n_0$ the \emph{strict inequality $P^{\pi}_{f^n}(1) < P_{f^n}(1)$} holds, where $f^n:X^n \to Y^n$. In the final section we pose a question whether a ``Hilali"-type inequality $HP^{\pi}_X(r_X) \leqq P_X(r_X)$ holds for a rationally hyperbolic space $X$, provided the the homotopical Hilbert--Poincare series $HP^{\pi}_X(r_X)$ converges at the radius $r_X$ of convergence.
\end{abstract}

\section{Introduction}

The most important and fundamental topological invariant in geometry and topology is the Euler--Poincar\'e characteristic $\chi(X)$,
which is 
the alternating sum of the Betti numbers $\dim H_i(X;\Q)$:
$$\chi(X):= \sum_{i \geqq 0} (-1)^i\dim H_i(X;\Q) ,$$
provided that each $\dim H_i(X;\Q)$ and $\chi(X)$ are both finite.
Similarly, for a topological space whose fundamental group is an Abelian group one can define the \emph{homotopical Betti number} $\dim (\pi_i(X)\otimes \Q)$ where $i\geqq 1$ and the \emph{homotopical Euler--Poincar\'e characteristic}:
$$\chi^{\pi}(X):= \sum_{i \geqq 1}  (-1)^i\dim (\pi_i(X)\otimes \Q),$$
provided that each $\dim (\pi_i(X)\otimes \Q)$ and $\chi^{\pi}(X)$ are both finite.
The Euler--Poincar\'e characteristic is the special value of the Poincar\'e polynomial $P_X(t)$ at $t=-1$ and the homotopical Euler--Poincar\'e characteristic is the special value of the homotopical Poincar\'e polynomial $ P^{\pi}_X(t)$ at $t=-1$:
$$P_X(t):= \sum_{i \geqq 0}  t^i \dim H_i(X;\Q), \quad \chi(X) = P_X(-1),$$
$$ P^{\pi}_X(t):= \sum_{i \geqq 1} t^i \dim (\pi_i(X)\otimes \Q), \quad \chi^{\pi}(X) = P^{\pi}_X(-1).$$
Since we consider polynomials, besides the requirement that $\dim H_i(X;\Q) $ and $\dim \left (\pi_i(X) \otimes \Q \right) $ are each finite, we assume that there exist integers $n_0$ and $m_0$ such that $H_i(X;\Q) =0$ for $\forall i >n_0$ and $\pi_j(X)\otimes \Q=0$ for $\forall j >m_0$, which are equivalent to requiring that
$$\dim H_*(X;\Q ) := \sum_{i \geqq 0} \dim H_i(X;\Q) < \infty, \quad 
\dim (\pi_*(X)\otimes \Q ) : = \sum_{i \geqq 1} \dim (\pi_i(X) \otimes \Q) < \infty.$$
Such a space $X$ is called \emph{rationally elliptic}. If we have
$$\dim H_*(X;\Q ) < \infty,  \quad 
\dim (\pi_*(X)\otimes \Q ) =\infty,$$ then such a space $X$ is called \emph{rationally hyperbolic}, because it follows (see \cite[Theorem 33.2]{FHT}) that there exist some $C >1$ and some positive integer $K$ such that 
$$\sum_{i\geqq 2}^k \dim (\pi_i(X) \otimes \Q) \geqq C^k, \quad k \geqq K.$$

From now on, unless otherwise stated, any topological space is assumed to be simply connected and of finite type (over $\mathbb Q$), i.e., the rational homology group is finitely generated for every dimension, $\dim H_i(X; \mathbb Q) < \infty$, which implies that $\dim \left (\pi_i(X) \otimes  \mathbb Q) \right ) < \infty$ because it is well-known that a simply connected space has finitely generated homology groups in every dimension \emph{if and only if} it has finitely generated homotopy groups in every dimension (e.g., see \cite[16 Corollary, p.509]{Sp}). A very simple example of a non-simply connected space for which this statement does not hold is $S^2 \vee S^1$.

The well-known Hilali conjecture \cite{Hil} claims that if $X$ is a simply connected rationally elliptic space, then
$$\dim (\pi_*(X)\otimes \Q ) \leqq \dim H_*(X;\Q ), \quad \text{namely}, \quad P^{\pi}_X(1) \leqq P_X(1).$$ 
No counterexample to the Hilali conjecture has been so far found yet.

In \cite{Yo} the second named author proved that for a simply connected rationally elliptic space $X$ the Hilali conjecture always holds ``modulo product", i.e.,
there exists a positive integer $n_{0}$ such that for $\forall \, \, n \geqq n_{0}$
\begin{equation}\label{hil-pro}
\dim (\pi_*(X^n)\otimes \Q ) < \dim H_*(X^{n};\Q ), \, \text{i.e.,} \, P^{\pi}_{X^n}(1) < P_{X^n}(1).
\end{equation}
Here $X^n$ is the product $X^n =\underbrace{X \times \cdots \times X}_{n}$. 

In this paper we introduce the homological and homotopical Poincar\'e polynomials $P_f(t)$ and $P^{\pi}_f(t)$ of a continuous map $f:X \to Y$ and show that if $P_f(1) > 1$, i.e., there exists some integer $i >1$ such that $H_i(f;\mathbb Q):H_i(X;\mathbb Q) \to H_i(Y;\mathbb Q)$ is not injective, then there exists a positive integer $n_{0}$ such that for $\forall \, \, n \geqq n_{0}$
the strict inequality $P^{\pi}_{f^n}(1) < P_{f^n}(1) $ holds, where $f^n:X^n \to Y^n$ is defined component-wise by $(f^n)(x_1, \cdots, x_n):=(f(x_1), \cdots, f(x_n))$.  This result is a map version of the above result (\ref{hil-pro}). 

Hinted by the proof \cite{Yo} of $P^{\pi}_{X^n}(1) < P_{X^n}(1)$, we give a reasonable conjecture claiming that  if $P_f(1)=1$, then $P_f^{\pi}(1)=0$, in other words, if each homological homomorphism
 $H_i(f;\mathbb Q):H_i(X;\mathbb Q) \to H_i(Y;\mathbb Q)$ being \emph{injective} for $\forall i >1$ implies that each homotopical homomorphism $\pi_i(f;\mathbb Q):\pi_i(X)\otimes \mathbb Q \to \pi_i(Y) \otimes \mathbb Q$ is \emph{injective} $\forall i >1$. We remark that for this conjecture we assume that $X$ and $Y$ are rationally elliptic spaces
and that the conjecture is false if the homology rank of the target $Y$ is not finite, as shown by a counterexample later.
In fact, as seen in Conjecture \ref{injective}, for the above conjecture we assume that the map $f:X \to Y$ is \emph{a rationally elliptic map} (see Definition \ref{k-cok-el} below). Ellipticity of a map $f:X \to Y$ is a more lax condition than requiring $X$ and $Y$ to be rationally elliptic, in which case $f$ is certainly a rationally elliptic map.

In passing, we recall that the well-know Whitehead--Serre Theorem (e.g., see \cite[Theorem 8.6]{FHT}) claims that for simply connected spaces $X$ and $Y$, $H_i(f;\mathbb Q):H_i(X;\mathbb Q) \to H_i(Y;\mathbb Q)$ is \emph{isomorphic} for $\forall i >0$ if and only if $\pi_i(f) \otimes \mathbb Q:\pi_i(X)\otimes \mathbb Q \to \pi_i(Y) \otimes \mathbb Q$ is \emph{isomorphic} $\forall i >1$. In \cite{Yo}, to show the above (\ref{hil-pro}), we need to show that if $P_X(1)=1$, then $P^{\pi}_X(1)=0$, for which we use this Whitehead--Serre Theorem.

In the final section we discuss the case of hyperbolic spaces a bit. For a hyperbolic space $X$ we have the homotopical Hilbert--Poincar\'e series $HP^{\pi}_X(t)$ instead of the polynomial $P^{\pi}_X(t)$.
It is known (see \cite{Fe}) that the radius $r_X$ of convergence of $HP^{\pi}_X(t)$ is \emph{less than $1$}. It is in general well-known that if $r$ denotes the radius of convergence of a power series $P(t)$, then whether $P(r)$ converges or not is case-by-case. So, when $HP^{\pi}_X(r_X)$ does converge, it seems to be an interesting question if the following holds or not:
$$HP^{\pi}_X(r_X) \leqq P_X(r_X),$$
which could be called \emph {``a Hilali conjecture in the hyperbolic case"}.

\section{Homological and homotopical Poincar\'e polynomials of a map}
Let $f:X \to Y$ be a continuous map of simply connected spaces $X$ and $Y$ of finite type. For the homomorphisms
$$H_i(f; \mathbb Q):H_i(X;\mathbb Q) \to H_i(Y;\mathbb Q),\quad \pi_i(f)\otimes \mathbb Q:\pi_i(X) \otimes \mathbb Q \to \pi_i(Y) \otimes \mathbb Q,$$
we have the following exact sequences of finite dimensional $\mathbb Q$-vector spaces:
\begin{equation}\label{hkk}
0 \to \op{Ker} H_i(f; \mathbb Q) \to H_i(X; \mathbb Q) \to H_i(Y; \mathbb Q) \to \op{Coker} H_i(f;\mathbb Q) \to 0 \quad \forall i \geqq 0,
\end{equation}
\begin{equation}\label{pikk}
0 \to \op{Ker}(\pi_i(f)\otimes \mathbb Q) \to \pi_i(X)\otimes \mathbb Q \to \pi_i(Y)\otimes \mathbb Q \to \op{Coker}(\pi_i(f)\otimes \mathbb Q) \to 0 \quad \forall i \geqq 2.
\end{equation}
Here we recall that $\op{Coker}(T):= B/\op{Im}(T)$ for a linear map $T: A \to B$ of vector spaces. 

Since $X$ and $Y$ are simply connected, they are path-connected as well (by the definition of simply connectedness), thus we have
$$\xymatrix{ 
\mathbb Q \cong H_0(X;\mathbb Q) \ar[r]^{f_*}_{\cong} & H_0(Y;\mathbb Q) \cong \mathbb Q,
}
$$
so $\op{Ker} H_0(f; \mathbb Q) = \op{Coker}H_0(f; \mathbb Q)=0$.  
It follows from (\ref{hkk}) and (\ref{pikk}) that we get the following equalities:
\begin{equation}\label{d-hkk}
\op{dim}(\op{Ker} H_i(f; \mathbb Q)) - \op{dim} H_i(X; \mathbb Q)  + \op{dim} H_i(Y; \mathbb Q)  - \op{dim}(\op{Coker}H_i(f; \mathbb Q)) = 0 \quad \forall i \geqq 2,
\end{equation}
\begin{equation}\label{d-pikk}
\op{dim}(\op{Ker}(\pi_i(f)\otimes \mathbb Q)) - \op{dim}(\pi_i(X) \otimes \mathbb Q) + \op{dim}(\pi_i(Y) \otimes \mathbb Q) - \op{dim}(\op{Coker}(\pi_i(f)\otimes \mathbb Q)) = 0 \quad \forall i \geqq 2.
\end{equation}

\begin{defn}\label{k-cok-el} Let $f:X \to Y$ be a continuous map of simply connected spaces $X$ and $Y$.
\begin{enumerate}
\item If $\op{dim} \left (\op{Ker} H_*(f; \mathbb Q) \right):= \sum_i \op{dim} \left (\op{Ker} H_i(f; \mathbb Q) \right) < \infty$ and $\op{dim} \left (\op{Ker} (\pi_*(f) \otimes \mathbb Q) \right) := \sum_i \op{dim}  \left (\op{Ker}  (\pi_i(f) \otimes \mathbb Q) \right ) < \infty$,
then $f$ is called \emph{rationally elliptic with respect to kernel}.
\item If $\op{dim} \left (\op{Coker} H_*(f; \mathbb Q) \right) :=\sum_i \op{dim} \left ( \op{Coker} H_i(f; \mathbb Q)\right )< \infty,$ and $\op{dim} \left (\op{Coker}  (\pi_*(f) \otimes \mathbb Q) \right ):=\sum_i \op{dim} \left (\op{Coker}  (\pi_i(f) \otimes \mathbb Q) \right )< \infty$, 
then $f$ is called \emph{rationally elliptic with respect to cokernel}.
\item If the map $f$ is rationally elliptic with respect to both kernel and cokernel, $f$ is called \emph{rationally elliptic}.
\end{enumerate}
\end{defn}
\begin{rem} Let $f:X \to Y$ be a continuous map of simply connected spaces $X$ and $Y$.
\begin{enumerate}
\item If $X$ is rationally elliptic, then $f$ is rationally elliptic with respect to kernel.
\item If $Y$ is rationally elliptic, then $f$ is rationally elliptic with respect to cokernel.
\item If both $X$ and $Y$ are rationally elliptic, then $f$ is rationally elliptic.
\end{enumerate}
\end{rem}
In this connection we also give definitions of ``hyperbolic" one corresponding to each above.
\begin{defn} Let $f:X \to Y$ be a continuous map of simply connected spaces $X$ and $Y$.
\begin{enumerate}
\item If $\op{dim} \left(\op{Ker} H_*(f; \mathbb Q) \right) < \infty$ and $\op{dim} \left (\op{Ker} (\pi_*(f) \otimes \mathbb Q) \right) = \infty$,
then $f$ is called \emph{rationally hyperbolic with respect to kernel}.
\item If $\op{dim} \left (\op{Coker} H_*(f; \mathbb Q) \right)  < \infty$ and $\op{dim} \left (\op{Coker} (\pi_*(f) \otimes \mathbb Q) \right) = \infty$, 
then $f$ is called \emph{rationally hyperbolic with respect to cokernel}.
\item If the map $f$ is rationally hyperbolic with respect to both kernel and cokernel, $f$ is called \emph{rationally hyperbolic}.
\end{enumerate}
\end{defn}
\begin{rem}\label{hyperbol-rem} Let $f:X \to Y$ be a continuous map of simply connected spaces $X$ and $Y$.
\begin{enumerate}
\item If $f:X \to Y$ is rationally hyperbolic with respect to kernel, then the homotopy rank of $X$ is infinite.
\item If $f:X \to Y$ is rationally hyperbolic with respect to cokernel, then the homotopy rank of $Y$ is infinite.
\item If $f:X \to Y$ is rationally hyperbolic, then the homotopy rank of $X$ and that of $Y$ are both infinite.
\end{enumerate}
\end{rem}
Motivated by the definition of Poincar\'e polynomials of topological spaces, it is reasonable to make the following definitions:
\begin{defn} Let $f:X \to Y$ be a rationally ellitpic map of simply connected spaces $X$ and $Y$.
\begin{enumerate}
\item (the homological ``Kernel" Poincar\'e polynomial of a map $f$)
$$\op{Ker} P_f(t):= \sum_{i\geqq 2}  \op{dim}(\op{Ker} H_i(f; \mathbb Q) ) t^i.$$
\item  (the homotopical ``Kernel" Poincar\'e polynomial of a map $f$)
$$\op{Ker}P^{\pi}_f(t):= \sum_{i\geqq 2}  \op{dim}(\op{Ker}(\pi_i(f)\otimes \mathbb Q)) t^i.$$
\item (the homological ``Cokernel" Poincar\'e polynomial of a map $f$)
$$\op{Cok}P_f(t):= \sum_{i\geqq 2}  \op{dim}(\op{Coker} H_i(f; \mathbb Q) ) t^i.$$
\item  (the homotopical ``Cokernel" Poincar\'e polynomial of a map $f$)
$$\op{Cok}P^{\pi}_f(t):= \sum_{i\geqq 2}  \op{dim}(\op{Coker}(\pi_i(f)\otimes \mathbb Q)) t^i.$$
\end{enumerate}
\end{defn}
With these definitions, if $X$ and $Y$ are both rationally elliptic, then it follows from (\ref{d-hkk}) and (\ref{d-pikk}) that we get the following equalities:
\begin{equation}\label{kppc}
\op{Ker}P_f(t) - P_X(t) + P_Y(t) - \op{Cok}P_f(t) =0,
\end{equation}
\begin{equation}\label{pi-kppc}
\op{Ker}P^{\pi}_f(t) - P^{\pi}_X(t) + P^{\pi}_Y(t) - \op{Cok}P^{\pi}_f(t) =0.
\end{equation}
If $H_i(f; \mathbb Q)$ and $\pi_i(f) \otimes \mathbb Q$ are surjective for $\forall i  \geqq 2$, then 
$\op{Coker} H_i(f; \mathbb Q) = \op{Coker} (\pi_i(f)\otimes \mathbb Q)=0$, 
thus we have
\begin{equation}
\op{Ker}P_f(t) - P_X(t) + P_Y(t)=0,
\end{equation}
\begin{equation}
\op{Ker}P^{\pi}_f(t) - P^{\pi}_X(t) + P^{\pi}_Y(t)=0.
\end{equation}
In particular, when $Y$ is contractible, since $P_Y(t)=1$ and $P^{\pi}_Y(t)=0$, we have
\begin{equation}\label{cont-h}
P_X(t) = 1 + \op{Ker}P_f(t)
\end{equation}
\begin{equation}\label{cont-pi}
P^{\pi}_X(t) = \op{Ker} P^{\pi}_f(t).
\end{equation}

In this paper we focus mainly on continuous rationally elliptic maps with respect to kernel. Let 
$f:X \to Y$ be a continuous rationally elliptic map with respect to kernel of simply connected spaces $X$ and $Y$ and we define the following:
\begin{defn}[Homological Poincar\'e polynomial of a map]
\begin{equation}
P_f(t):= 1 + \op{Ker} P_f(t) = 1 + \sum_{i\geqq 2} t^i \op{dim} \left (\op{Ker} H_i(f; \mathbb Q) \right).
\end{equation}
\end{defn}
\begin{defn}[Homotopical Poincar\'e polynomial of a map]
\begin{equation}
P^{\pi}_f(t):=  \op{Ker} P^{\pi}_f(t) = \sum_{i\geqq 2} t^i \op{dim} \left (\op{Ker}  (\pi_i(f) \otimes \mathbb Q) \right).
\end{equation}
\end{defn}

From (\ref{cont-h}) and (\ref{cont-pi}), if $Y$ is contractible, then we have
\begin{equation}
P_f(t)= P_X(t), \quad P^{\pi}_f(t)= P^{\pi}_X(t).
\end{equation}
\section{The relative Hilali conjecture on products of maps}
In our previous paper \cite{YaYo} we made the following conjecture, called \emph{a relative Hilali conjecture}
\begin{con} For a continuous map $f:X \to Y$ of simply connected elliptic spaces $X$ and $Y$, $P^{\pi}_f(1) \leqq P_f(1)$ holds. Namely the following inequality holds:
$$\sum_{i\geqq 2} \op{dim} \left (\op{Ker}  (\pi_i(f) \otimes \mathbb Q) \right) \leqq  1 + \sum_{i\geqq 2} \op{dim} \left (\op{Ker} H_i(f; \mathbb Q) \right ).$$
\end{con}

When $Y$ is a point or contractible, the above relative Hilali conjecure 
is nothing but the following well-known Hilali conjecture \cite{Hil}:
\begin{con} For a simply connected elliptic space $X$, $P^{\pi}_X(1) \leqq P_X(1)$ holds. Namely the following inequlaity holds:
$$\sum_{i\geqq 2} \op{dim}(\pi_i(X) \otimes \mathbb Q) \leqq  1 + \sum_{i\geqq 2} \op{dim} H_i(X; \mathbb Q).$$
\end{con}
\begin{rem} We note that in the Hilali conjecture the inequality $\leqq$ cannot be replaced by the strict inequality $<$. Indeed, for example, if $X=S^{2k}$ the even dimensional sphere, we have
$$\pi_i(S^{2k}) \otimes \Q 
=\begin{cases} 
\Q & \, i=2k\\
 \Q & \, i=4k-1\\
\, 0  & \,  i\not =2k, 4k-1.
\end{cases}
$$
Thus we have $P^{\pi}_{S^{2k}}(t) = t^{4k-1} + t^{2k} \, \text{and} \, P_{S^{2k}}(t) = t^{2k} +1.$ Hence $P^{\pi}_{S^{2k}}(1) = P_{S^{2k}}(1) = 2.$
\end{rem}
In \cite{ZCH} (cf. \cite{CHHZ}) A. Zaim, S. Chouingou and M. A. Hilali have proved the above relative Hilali conjecture in some cases.

Since we define the notion of rationally elliptic map with respect to kernel in the previous section, in the original version of this paper we speculated that the above relative Hilali conjecture could be furthermore generalized as follows:

\emph{``(A generalized relative Hilali conjecture): Let $f:X \to Y$  be a continuous rationally elliptic map with respect to kernel of simply connected spaces $X$ and $Y$. Then
$P^{\pi}_f(1) \leqq P_f(1)$ holds." }

It turns out that this conjecture is false due to the following counterexample, which was given by the referee:
\begin{ex}\label{c-ex}
Consider the following map
$$f: S^4 \times S^6 \to K(\mathbb Q, 4) \times K(\mathbb Q, 6)$$
which is defined by $f:= a \times b$. Here $a:S^4 \to K(\mathbb Q, 4)$ is such that $[a] \in [S^4, K(\mathbb Q, 4)]=H^4(S^4, \mathbb Q)= \mathbb Q$ is a generator and similar for $b: S^6 \to K(\mathbb Q, 6)$.
Then we have 
$$ P^{\pi}_f(1) = \dim (\op{Ker} (\pi_*(f) \otimes \mathbb Q)) = 2, \quad P_f(1)=1, \, \, \text{i.e.,} \, \, \dim \left (\op{Ker} (H_*(f; \mathbb Q)) \right)= 0.$$
Thus $P^{\pi}_f(1) \not \leqq P_f(1)$. 
\end{ex}
Here we note that in this counterexample $\dim H_*(K(\mathbb Q, 4) \times K(\mathbb Q, 6); \mathbb Q) =\infty$ although we have that $\dim \left (\pi_*(K(\mathbb Q, 4) \times K(\mathbb Q, 6)) \otimes \mathbb Q \right) <\infty$. So, if in the above generalized Hilali conjecture we add another requirement that the homology rank of the target $Y$ is finite, 
then it follows from (\ref{kppc}) with $t=1$  
that the homology rank of the source $X$ has to be automatically finite. 
If we furthermore require that the target $Y$ is rationally elliptic, then it follows from (\ref{kppc}) and (\ref{pi-kppc}) with $t=1$ that the source $X$ has to be automatically also rationally elliptic, thus it becomes the original relative Hilali conjecture. So, we would like to pose the following slightly modified conjecture:
\begin{con} (A generalized relative Hilali conjecture) \label{grhc}
Let $f:X \to Y$  be a continuous rationally elliptic map with respect to kernel of simply connected spaces $X$ and $Y$. 
If the homology rank of the target $Y$ is finite, then
$P^{\pi}_f(1) \leqq P_f(1)$ holds. 
\end{con}

In \cite{Yo} (cf. \cite{Yo2}) the second named author has proved the following
\begin{thm}[\emph{Hilali conjecture ``modulo product"}]\label{hilali-pro} Let $X$ be a  
rationally elliptic space such that its fundamental group is an Abelian group. Then there exists some integer $n_0$ such that for $\forall \, \, n\geqq n_0$ the strict inequality $ P^{\pi}_{X^n}(1) < P_{X^n}(1)$ holds, i.e.,
\begin{equation}\label{pro-hilali}
\op{dim} \left (\pi_*(X^n)\otimes \Q \right ) < \op{dim} H_*(X^{n};\Q ). 
\end{equation}
\end{thm}
In this section, as a ``map version" of the above theorem, we show the following theorem, in which we do not require that the homology rank of the target $Y$ is finite (hence the homology rank of the source $X$ is automatically finite as explained above), instead we require that the homology rank of the source $X$ is finite:
\begin{thm}[A generalized relative Hilali conjecture ``modulo product"]\label{main} Let $f:X \to Y$ be a continuous rationally elliptic map with respect to kernel of simply connected spaces $X$ and $Y$ such that 
the homology rank of the source $X$ is finite. If $P_f(1) > 1$, i.e., there exists some integer $i$ such that $H_i(f;\mathbb Q):H_i(X;\mathbb Q) \to H_i(Y;\mathbb Q)$ is not injective, then there exists some integer $n_0$ such that for $\forall \, \, n\geqq n_0$ the strict inequality $P^{\pi}_{f^n}(1) < P_{f^n}(1)$ holds, i.e., 
\begin{equation}\label{pro-rel-hilali}
\sum_{i\geqq 2} \op{dim} \left (\op{Ker} (\pi_i(f^n) \otimes \mathbb Q)  \right )<  1 + \sum_{i\geqq 2} \op{dim} \left (\op{Ker} H_i(f^n; \mathbb Q)  \right ).
\end{equation}
\end{thm}
\begin{rem} Note that if $Y$ is contractible, then the formula (\ref{pro-rel-hilali}) becomes the formula (\ref{pro-hilali}).
In this case, the above requirement $P_f(1) > 1$ becomes $P_X(1) >1$, which can be dropped, namely $P_X(1)=1$ can be allowed. As explained in the introduction, by using Whitehead--Serre Theorem we can show that if $P_X(1)=1$, then $P_X^{\pi}(1)=0$. Thus for $\forall n \geqq n_0=1$ we have $0=n(P^{\pi}_X(1))= P^{\pi}_{X^n}(1) < P_{X^n}(1) = (P_X(1))^n=1$.
\end{rem}
\begin{rem} 
In the above Theorem \ref{main} we pose the condition that the homology rank of the source $X$ is finite. This is needed so that any product $f^n:X^n \to Y^n$ is also rationally elliptic with respect to kernel, thus we can consider the Poincar\'e polynomial $P_{f^n}(t)$ and a finite integer $P_{f^n}(1)$. The crucial condition is that $P_f(1) > 1$, i.e., $\dim \left (\op{Ker} (H_*(f; \mathbb Q)) \right) \not = 0$ unlike the above counterexample Example \ref{c-ex}. 
If in the theorem we drop the condition that the homology rank of the source $X$ is finite, then $\sum_{i\geqq 2} \op{dim} \left (\op{Ker} (H_i(f^n; \mathbb Q) ) \right ) =\infty$ can happen and in this case $P_{f^n}(t)$ becomes \emph{a Hilbert--Poincar\'e power series $HP_{f^n}(t)$, not a polynomial}. In this case the above strict inequality (\ref{pro-rel-hilali}) automatically holds because the left-hand side is always finite and the right-hand-side is $\infty$. In this sense, we could drop the condition that the homology rank of the source $X$ is finite, if we are allowed to understand $P_{f^n}(t)$ as the Hilbert--Poincar\'e series $HP_{f^n}(t)$ for the obvious strict inequality $P^{\pi}_{f^n}(1) < P_{f^n}(1) =\infty$.
\end{rem}


A key ingredient for the proof of the above Theorem \ref{hilali-pro} is the following 
\emph{multiplicativity} of the homological Poincar\'e polynomial and \emph{additivity} of the homotopy Poincar\'e polynomial:
\begin{equation}\label{x+}
P_{X \times Y}(t) = P_X(t) \times P_Y(t), \quad P^{\pi}_{X \times Y}(t) = P^{\pi}_X(t) + P^{\pi}_Y(t).
\end{equation} 
In order to prove the above Theorem \ref{main} first we show the following ``map version" of the above multiplicativity and additivity (\ref{x+}): 
\begin{pro}\label{prop-multi} For two rationally elliptic maps with respect to kernels $f_1:X_1 \to Y_1, f_2:X_2 \to Y_2$, where $X_i, Y_i (i=1,2)$ are simply connected spaces such that 
both $X_1$ and $X_2$ have the finite homology rank 
, we have the following formulas:
\begin{enumerate}
\item $P^{\pi}_{f_1 \times f_2}(t) = P^{\pi}_{f_1}(t) + P^{\pi}_{f_2}(t),$ for $\forall t$
\item $P_{f_1}(t) \times P_{f_2}(t) \leqq P_{f_1 \times f_2}(t)$ for $\forall t \geqq 0$.
\end{enumerate}
\end{pro}
\begin{proof}
The proof is straightforward, but we give a proof for the sake of completeness.

First we observe that
$\pi_i(f_1 \times f_2) \otimes \mathbb Q: \pi_i(X_1 \times X_2) \otimes \mathbb Q \to \pi_i(Y_1 \times Y_2) \otimes \mathbb Q$ is the same as
$$(\pi_i(f_1) \otimes \mathbb Q) \oplus (\pi_i(f_2) \otimes \mathbb Q): (\pi_i(X_1) \otimes \mathbb Q) \oplus  (\pi_i(X_2) \otimes \mathbb Q) \to (\pi_i(Y_1) \otimes \mathbb Q) \oplus  (\pi_i(Y_2) \otimes \mathbb Q). $$
Hence 
$$\op{Ker}  \bigl(\pi_i(f_1 \times f_2) \otimes \mathbb Q \bigr ) = \op{Ker} \bigl(\pi_i(f_1) \otimes \mathbb Q \bigr) \oplus \op{Ker}  \bigl(\pi_i(f_2) \otimes \mathbb Q \bigr),$$
which implies 
\begin{equation}\label{eq-pi}
\op{dim} \left (\op{Ker}  \bigl(\pi_i(f_1 \times f_2) \otimes \mathbb Q \bigr ) \right )= \op{dim} \left (\op{Ker}  \bigl(\pi_i(f_1) \otimes \mathbb Q \bigr)  \right) + \op{dim} \left (\op{Ker}  \bigl(\pi_i(f_2) \otimes \mathbb Q \bigr) \right ).
\end{equation}
Thus $\op{dim} \left (\op{Ker}  \bigl(\pi_*(f_1) \otimes \mathbb Q \bigr)  \right) <\infty $ and $\op{dim} \left (\op{Ker}  \bigl(\pi_*(f_2) \otimes \mathbb Q \bigr) \right ) <\infty$ imply that 
$$\op{dim} \left (\op{Ker}  \bigl(\pi_*(f_1 \times f_2) \otimes \mathbb Q \bigr ) \right ) < \infty.$$
Since the homology rank of $X_i \, (i=1,2)$ is finite, i.e., $\op{dim} H_*(X_i;\mathbb Q) < \infty \, (i=1,2)$, we have that $\op{dim} \left (\op{Ker} H_*(f_1 \times f_2;\mathbb Q \bigr ) \right )<\infty$, because $H_*(X_1 \times X_2; \mathbb Q) \cong H_*(X_1; \mathbb Q) \otimes H_*(X_2; \mathbb Q)$, thus $\op{dim} H_*(X_1 \times X_2;\mathbb Q)<\infty$.
Therefore the product $f_1 \times f_2: X_1 \times  X_2 \to Y_1 \times Y_2$ is also a rationally elliptic map with respect to kernel.

(1) From (\ref{eq-pi}) above we get
\begin{align*}
P^{\pi}_{f_1 \times f_2}(t) & = \sum_{i\geqq 2} t^i \op{dim} \left (\op{Ker}  \bigl(\pi_i(f_1 \times f_2) \otimes \mathbb Q \bigr ) \right )\\
& = \sum_{i\geqq 2} t^i \op{dim} \left (\op{Ker}  \bigl(\pi_i(f_1) \otimes \mathbb Q \bigr) \right) + \sum_{i\geqq 2} t^i  \op{dim} \left (\op{Ker}  \bigl(\pi_i(f_2) \otimes \mathbb Q \bigr) \right) \\
& = P^{\pi}_{f_1}(t) + P^{\pi}_{f_2}(t).
\end{align*}

(2) $H_i(f_1 \times f_2;\mathbb Q): H_i(X_1 \times X_2; \mathbb Q) \to H_i(Y_1 \times Y_2; \mathbb Q)$ can be expressed as follows by K\"unneth theorem:
$$H_i(f_1 \times f_2;\mathbb Q): \sum_{i=j+k} H_j(X_1;\mathbb Q) \otimes  H_k(X_2;\mathbb Q) \to \sum_{i=j+k}H_j(Y_1;\mathbb Q) \otimes  H_k(Y_2;\mathbb Q)$$
Since $X_i$ and $Y_i$ ($i=1, 2$) are simply connected, the products $X_1 \times X_2$ and $Y_1 \times Y_2$ are also simply connected. Hence $ \op{Ker} H_0(f_1 \times f_2; \mathbb Q) = \op{Ker} H_1(f_1 \times f_2; \mathbb Q)  =0.$ 

For $i \geqq 2$, we have the following inequality (*)
\begin{align*} \label{sum}
\Bigl ( \op{Ker} H_i(f_1;\mathbb Q)  & \otimes H_0(X_2;\mathbb Q) \Bigr ) \oplus  \Bigl (H_0(X_1;\mathbb Q) \otimes \op{Ker} H_i(f_2;\mathbb Q) \Bigr) \oplus  \\
& \hspace{2cm}   \sum_{i=j+k, j\geqq 2, k \geqq 2} \op{Ker} H_j(f_1;\mathbb Q)  \otimes \op{Ker} H_k(f_2;\mathbb Q) \\
& \hspace{2cm} \subset \op{Ker} H_i(f_1 \times f_2; \mathbb Q).
\end{align*}
Clearly 
$$\sum_{i=j+k, j\geqq 2, k \geqq 2} \op{Ker} H_j(f_1;\mathbb Q) \otimes H_k(X_2; \mathbb Q) \, \,  +  \sum_{i=j+k, j\geqq 2, k \geqq 2} H_j(X_1; \mathbb Q) \otimes \op{Ker} H_k(f_2;\mathbb Q) $$
is also contained in $\op{Ker} H_i(f_1 \times f_2; \mathbb Q)$, and furthermore
probably one could obtain a complete description of $\op{Ker} H_i(f_1 \times f_2; \mathbb Q)$, but for our purpose we do not need to do so and the above inequality (*) is sufficient. The dimension of the above is equal to the following: for $i \geqq 2$
\begin{align*}
\op{dim} \left ( \op{Ker} H_i(f_1;\mathbb Q) \right ) &  + \op{dim} \left (\op{Ker} H_i(f_2;\mathbb Q)  \right ) \\
& + \sum_{i=j+k, j\geqq 2, k \geqq 2}  \op{dim} \left (\op{Ker} H_j(f_1;\mathbb Q) \right ) \times \op{dim}\left (\op{Ker} H_k(f_2;\mathbb Q)  \right )\\
& \leqq \op{dim} \left (\op{Ker} H_i(f_1 \times f_2; \mathbb Q) \right ). 
\end{align*}
Therefore we have that for each $i \geqq 2$ and $t \geqq 0$:
\begin{align*}
t^i\op{dim} \left ( \op{Ker} H_i(f_1;\mathbb Q) \right ) &+ t^i\op{dim} \left ( \op{Ker} H_i(f_2;\mathbb Q) \right ) \\
& + \sum_{i=j+k, j\geqq 2, k \geqq 2}  t^j\op{dim} \left (\op{Ker} H_j(f_1;\mathbb Q)  \right ) \times t^k \op{dim} \left (\op{Ker} H_k(f_2;\mathbb Q)  \right )\\
& \leqq t^i\op{dim} \left ( \op{Ker} H_i(f_1 \times f_2; \mathbb Q) \right ). 
\end{align*}
Therefore we have 
\begin{align*}
P_{f_1 \times f_2}(t) &= 1 + \sum_{i\geqq 2} t^i \op{dim} \left (\op{Ker} H_i(f_1 \times f_2; \mathbb Q) \right )\\
& \geqq  1 + \sum_{i\geqq 2} t^i \op{dim} \left (\op{Ker} H_i(f_1; \mathbb Q) \right )+ \sum_{i\geqq 2} t^i \op{dim} \left ( \op{Ker} H_i(f_2; \mathbb Q) \right )\\
& \hspace{1cm} 
+ \sum_{i\geqq 4} \Bigl ( \sum_{i=j+k, j\geqq 2, k\geqq 2} t^j\op{dim} \left (\op{Ker} H_j(f_1; \mathbb Q) \right )\times t^k\op{dim} \left (\op{Ker} H_k(f_2; \mathbb Q)\right ) \Bigr )  \\
& = \Bigl  (1 + \sum_{j\geqq 2} t^j \op{dim} \left (\op{Ker} H_j(f_1; \mathbb Q) \right ) \Bigr ) \times \Bigl (1 + \sum_{k\geqq 2} t^k \op{dim} \left (\op{Ker} H_k(f_2; \mathbb Q) \right )\Bigr)\\
& = P_{f_1}(t)  \times P_{f_2}(t). 
\end{align*}
Hence we have $ P_{f_1}(t)  \times P_{f_2}(t) \leqq P_{f_1 \times f_2}(t)$ for $\forall t \geqq 0$.
\end{proof}
\begin{rem} The equality $ P_{f_1}(t)  \times P_{f_2}(t) = P_{f_1 \times f_2}(t)$ does not hold in general. However, in order to prove Theorem \ref{main} the above inequality (2) of Proposition \ref{prop-multi} is sufficient . 

\end{rem}
\begin{cor}\label{cor} Let $f:X \to Y$ be a continuous rationally elliptic map with respect to kernel of simply connected spaces $X$ and $Y$ such that the homology rank of $X$ is finite. 
Then we have 
\begin{enumerate}
\item $P^{\pi}_{f^n}(t) = n \bigl (P^{\pi}_{f}(t) \bigr)$ for $\forall t$
\item $ \bigl (P_{f}(t) \bigr)^n \leqq P_{f^n}(t)$ for $\forall t \geqq 0$.
\end{enumerate}
\end{cor}
\begin{rem} Note that in (2) of Corollary \ref{cor} we do need $\forall t \geqq 0$. 
\end{rem}
\begin{cor}\label{cor-1} Let the setup be as in Proposition \ref{prop-multi}. Suppose that $P^{\pi}_{f_i}(1) \leqq P_{f_i}(1) \, (i=1,2)$. Then  $P^{\pi}_{f_1 \times f_2}(1) \leqq P_{f_1 \times f_2}(1)$ in the following cases:
\begin{enumerate}
\item $P_{f_i}(1) \geqq 2$ for $i=1,2$,
\item $P^{\pi}_{f_1}(1) = 0$ or $P^{\pi}_{f_2}(1) = 0$.
\end{enumerate}
In particular, if the relative Hilali conjecture holds for $f_1$ and $f_2$, then it also holds for the product $f_1 \times f_2$ in the above two cases.
\end{cor}
\begin{proof} 
First we note that $P_{f_i}(1) \geqq 1 \, (i=1,2)$ by the definition.
\begin{enumerate}
\item 
\begin{align*}
P^{\pi}_{f_1 \times f_2}(1) & = P^{\pi}_{f_1}(1) + P^{\pi}_{f_2}(1) \\
& \leqq P_{f_1}(1) + P_{f_2}(1) \\
& \leqq  P_{f_1}(1) \times P_{f_2}(1) \quad \text{(since $P_{f_i}(1) \geqq 2$ \, $(i=1,2)$ and (*) below)} \\
& \leqq P_{f_1 \times f_2}(1)
\end{align*}
(*) If $a, b \geqq 2$, then $ab -a -b = (a-1)(b-1)-1 \geqq 0$ because $a-1\geqq 1, b-1\geqq 1$.
\item For example, we let $P^{\pi}_{f_1}(1) = 0$. Then we have
\begin{align*}
P^{\pi}_{f_1 \times f_2}(1) & = P^{\pi}_{f_1}(1) + P^{\pi}_{f_2}(1) \\
& = P_{f_2}(1) \\
& \leqq  P_{f_1}(1) \times P_{f_2}(1) \quad \text{(since $P_{f_1}(1) \geqq 1$)} \\
& \leqq P_{f_1 \times f_2}(1)
\end{align*}
\end{enumerate}
\end{proof}
\begin{rem} 
The other cases which are not treated in Corollary \ref{cor-1} are the cases when at least one $P_{f_i}(1)=1$
and $P^{\pi}_{f_i}(1) \not =0 \, (i=1,2)$. For example, let $P_{f_1}(1)=1$. Then since 
$0 \not = P^{\pi}_{f_1}(1) \leqq P_{f_1}(1) =1$, we have $P_{f_1}(1)= P^{\pi}_{f_1}(1)=1.$
In this case at the moment we do not know whether $P^{\pi}_{f_1 \times f_2}(1) \leqq P_{f_1 \times f_2}(1)$ or not.
$P_{f}(1) =1$ means that $\op{Ker} H_*(f;\mathbb Q)=0$, i.e., $H_*(f;\mathbb Q): H_*(X;\mathbb Q) \to H_*(Y;\mathbb Q)$ is injective and $P^{\pi}_{f}(1) =1$ means that $\pi_*(f) \otimes \mathbb Q): \pi_*(X) \otimes \mathbb Q \to \pi_*(Y) \otimes \mathbb Q$ is \emph{not} injective, thus for this map $f_1$ the homological injectivity does not imply the homotopical injectivity.
If we could show that the homological injectivity implies the homotopical injectivity, i.e., $P_{f_1}(1) =1$ implies $P^{\pi}_{f_1}(1) =0$, which becomes the above second case (2). We will discuss this injectivity problem later.
\end{rem}
Now we give a proof of Theorem \ref{main}.
\begin{proof} 
If $P_{f}(1) >1$, i.e., there exists some integer $i \geqq 2$ such that the homomorphism $H_i(f;\mathbb Q): H_i(X; \mathbb Q) \to H_i(Y; \mathbb Q)$ is not injective, then, whatever the value $P^{\pi}_{f}(1)$ is, 
there exists some integer $n_0$ such that for $\forall n \geqq n_0$
$$n(P^{\pi}_{f}(1)) < (P_{f}(1))^n.$$
Indeed, since $P_{f}(1) >1$, we have that $\frac{1}{P_{f}(1)}<1$. It follows from an elementary fact in calculus ``$|r| <1 \Rightarrow \displaystyle \lim_{n \to \infty} nr^n=0$" that we have
$$\lim_{n \to \infty} n \Bigl (\frac{1}{P_f(1)} \Bigr)^n = 0.$$
Therefore, whatever the value $P^{\pi}_f(1)$ is, we obtain
$$\lim_{n \to \infty} n P^{\pi}_f(1)\Bigl (\frac{1}{P_f(1)} \Bigr)^n = \lim_{n \to \infty} \frac{nP^{\pi}_f(1)}{(P_f(1))^n} = 0.$$
Hence there exists some integer $n_0$ such that for $\forall n \geqq n_0$
$$\frac{nP^{\pi}_f(1)}{(P_f(1))^n} < 1,$$
which implies, using (2) of Corollary \ref{cor}, that
$$P^{\pi}_{f^n}(1) = n(P^{\pi}_{f}(1)) < (P_{f}(1))^n \leqq  P_{f^n}(1).$$
Therefore we can conclude that there exists some integer $n_0$ such that for all $n \geqq n_0$
$$P^{\pi}_{f^n}(1) <  P_{f^n}(1).$$
\end{proof}
As one can see, in the above proof, the requirement $P_f(1) > 1$ or the non-injectivity of $H_i(f;\mathbb Q):H_i(X;\mathbb Q) \to H_i(Y;\mathbb Q)$ for some $i$ is crucial. If we could show that the injectivity of each \emph{homological} homomorphisms $H_i(f;\mathbb Q):H_i(X;\mathbb Q) \to H_i(Y;\mathbb Q)$ would imply the injectivity of the \emph{homotopical} homomorphism $\pi_i(f) \otimes \mathbb Q :\pi_i(X)\otimes \mathbb Q \to \pi_i(Y) \otimes \mathbb Q$, then $0=P^{\pi}_f(1) < P_f(1)= 1,$
thus the above inequality would hold for $n_0=1$ and in fact, as we can see that for $\forall n \geqq n_0=1$ the inequality holds. But, as seen in the counterexample Example \ref{c-ex}, in the set-up of Theorem \ref{main}, the injectivity of each $H_i(f;\mathbb Q)$ \emph{does not} necessarily imply the injectivity of each $\pi_i(f) \otimes \mathbb Q$. In fact, the map $f: S^4 \times S^6 \to K(\mathbb Q, 4) \times K(\mathbb Q, 6)$ of Example (\ref{c-ex}) is \emph{not a continuous rationally elliptic map with respect to cokernel}. Furthermore we do have another counterexample: 
\begin{ex}
Consider the following canonical inclusion map
$$g: S^3 \vee S^3  \hookrightarrow S^3 \times S^3.$$
Then $\op{Ker} H_*(g;\mathbb Q)=0$, but $\dim \left (\op{Ker}(\pi_*(g)\otimes \mathbb Q) \right)=\infty$, thus the homological injectivity does not imply the homotopical injectivity. In this case we emphasize that $g$ is 
\emph{not a continuous rationally elliptic map with respect to kernel}.
\end{ex}

If $H_i(Y;\mathbb Q) = 0$ for $\forall i>0$, e.g., if $Y$ is contractible, then the injectivity of each homological homomorphisms $H_i(f;\mathbb Q):H_i(X;\mathbb Q) \to H_i(Y;\mathbb Q)$ means that $H_i(X;\mathbb Q)=0$. Furthermore, $\dim H_{*}(X; \Q) =1$ (for a pathconnected space $X$) is equivalent to $H_*(a_X; \mathbb Q) :H_*(X;\Q) \to 
H_*(pt)=\Q$ being an isomorphism, where $a_X:X \to pt$ is the map to a point. Thus it follows from the Whitehead--Serre Theorem \cite[Theorem 8.6]{FHT} that $(a_X)_*\otimes \Q :\pi_*(X)\otimes \Q \to \pi_*(pt)\otimes \Q =0$ is an isomorphism, hence $\pi_*(X)\otimes \Q =0$. Thus we get the injectivity of the homotopical homomorphism $\pi_i(f) \otimes \mathbb Q:\pi_i(X)\otimes \mathbb Q \to \pi_i(Y) \otimes \mathbb Q$. 

So, we would like to make the following conjecture, which we have been unable to resolve:
\begin{con}[``Injectivity conjecture"]\label{injective}
 Let $f:X \to Y$ be a continuous rationally elliptic map of simply connected 
  spaces $X$ and $Y$. 
 The injectivity of each homological homomorphism
 $H_i(f;\mathbb Q):H_i(X;\mathbb Q) \to H_i(Y;\mathbb Q)$ for $\forall i >1$ implies the injectivity of each homotopical homomorphism $\pi_i(f) \otimes \mathbb Q:\pi_i(X)\otimes \mathbb Q \to \pi_i(Y) \otimes \mathbb Q$ for $\forall i >1$. 
\end{con}
\begin{rem} In the original paper we made such a conjecture for \emph{a continuous map $f:X \to Y$ of simply connected elliptic spaces $X$ and $Y$}, which is surely a rationally elliptic map. Thus the above ``injectivity conjecture" is an extended version of the original conjecture.
\end{rem}

As a corollary of the above proof of Theorem \ref{main}, we can show that if $P_f(1) >1$, then for any $s >0$ there exists a positive integer $n(s)$ such that for $\forall n \geqq n(s)$
\begin{equation}\label{general one}
P^{\pi}_{f^n}(s) < P_{f^n}(s)
\end{equation}
because $P_f(1) >1$ implies $P_f(s)=1 + \sum_{i\geqq 2}\op{dim} \left (\op{Ker} H_i(f^n; \mathbb Q) \right) s^i >1$.
By the definition of $P_f(t)$ and $P^{\pi}_f(t)$ we have that $P_f(0)=1$ and $P^{\pi}_f(0)=0$. Hence for any integer $n \geqq 1$ we have that $0 = n(P^{\pi}_f(0))=P^{\pi}_{f^n}(0) < 1^n=P_f(0)^n = P_{f^n}(0)=1$ (whether $P_f(1) >1$ or not).
Therefore we get the following
\begin{cor}
If $P_f(1) >1$, then for any $s \geqq 0$ there exists a positive integer $n(s)$ such that for $\forall n \geqq n(s)$
$$P^{\pi}_{f^n}(s) < P_{f^n}(s).$$ 
\end{cor}

\section{A remark on the case of rationally hyperbolic maps}
Before finishing we give a remark about the case when $f:X \to Y$ is a rationally hyperbolic map with respect to kernel.

Since $f:X \to Y$ is rationally hyperbolic map with respect to kernel, as observed in Remark \ref{hyperbol-rem}, $X$ is rationally hyperbolic. Hence we have the homotopical Hilbert--Poincar\'e series and the homological Poincar\'e polynomial of $X$ and also those of $f:X \to Y$:
$$HP^{\pi}_X(t):= \sum_{i\geqq 2}^{\infty} \op{dim} \left (\pi_i(X) \otimes \mathbb Q \right ) t^i, \quad P_X(t)= 1 + \sum_{i\geqq 2} \op{dim} H_i(X; \mathbb Q) t^i$$
$$HP^{\pi}_f(t):= \sum_{i\geqq 2}^{\infty}  \op{dim} \left (\op{Ker} \bigl(\pi_i(f) \otimes \mathbb Q) \right ) t^i, \quad P_f(t)= 1+ \sum_{i\geqq 2} \op{dim} \left (\op{Ker} H_i(f; \mathbb Q) \right ) t^i $$
In \cite[Th\'eor\`eme 6.2.1]{Fe} Y. F\'elix showed that if $r_X$ denotes the radius of convergence of the above Hilbert--Poincar\'e series $HP^{\pi}_X(t)$ then $r_X <1$. For $t=0$ we have $HP^{\pi}_X(0)=0, P_X(0)=1$ and also $HP^{\pi}_f(0)=0, P_f(0)=1$, so we consider $r$ such that $0<r<r_f:=r(HP^{\pi}_f(t))$ the radius of convergence of the series $HP^{\pi}_f(t)$. Since we have
$$\op{dim} \left (\op{Ker} \left (\pi_i(f) \otimes \mathbb Q \right ) \right ) \leqq \op{dim} \bigl(\pi_i(X) \otimes \mathbb Q) $$
the convergence of $HP^{\pi}_X(r)$ implies the convergence of $HP^{\pi}_f(r)$, thus $r_X \leqq r_f<1$. Therefore, as a corollary of the proof of Theorem \ref{main}, we get the following corollary:
\begin{cor} Let $f:X \to Y$ be a rationally hyperbolic map with respect to kernel of simply connected spaces $X$ and $Y$. Let $P_f(1)>1$. Then for any $r$ such that $0<r<r_X$ there exists a positive integer $n(r)$ such that for $\forall n \geqq n(r)$
$$HP^{\pi}_{f^n}(r) <P_{f^n}(r).$$
\end{cor}
\begin{rem} Let $\alpha = \sum_{n \geqq 0} a^n t^n$ and $\beta =\sum_{n \geqq 0} b_n t^n$ be power series such that $0 \leqq a_n \leqq b_n$. Let $r(\alpha)$ and $r(\beta)$ be the radius of convergence of the power series $\alpha$ and $\beta$. Then they are not necessarily the same, in general $r(\beta) \leqq r(\alpha)$. Hence in the above corollary instead of $r_X$ we could take the radius $r_f$.
\end{rem}
Finally, let us consider the case when $Y$ is a point, i.e, we consider a rationally hyperbolic space $X$. We pose the following question:
\begin{qu}(a ``Hilali conjecture" in the hyperbolic case) \label{quest} Let $X$ be a rationally hyperbolic space. Let $r_X:= r(HP^{\pi}_X(t))$ be the radius of convergence as above. Suppose that $HP^{\pi}_X(t)$ converges at $r_X$, i.e., $HP^{\pi}_X(r_X) < \infty$. Does the following inequality hold?
$$HP^{\pi}_X(r_X) \leqq P_X(r_X).$$
\end{qu}
\begin{rem} We point out that some power series $p(x)$ converge at $x=r$ where $r=r(p(x))$ is the radius of convergence, but some do not. Here are some examples:
\begin{enumerate}
\item $p_1(x) = \sum_{n=1}^{\infty} \frac{x^n}{n^2} = 1 + x + \frac{x^2}{2^2} + \frac{x^3}{3^2} + \cdots$, $r(p_1(x) ) =1$ and $p_1(1) =   \sum_{n=1}^{\infty} \frac{1}{n^2}= \frac{\pi^2}{6}$ (This is nothing but the Basel problem.)
\item A modified version of $p_1(x)$ is the following: Let $d>0$.

\noindent
$p_2(x) = \sum_{n=1}^{\infty} \frac{(dx)^n}{n^2} = 1 + dx + \frac{(dx)^2}{2^2} + \frac{(dx)^3}{3^2} + \cdots$, $r(p_2(x)) =\frac{1}{d}$ and $p_2(\frac{1}{d}) =  \sum_{n=1}^{\infty} \frac{1}{n^2}= \frac{\pi^2}{6}.$

\item $p_3(x) =  \sum_{n=0}^{\infty} x^n = 1 + x + x^2 + \cdots$, $r(p_3(x)) =1$, but $p_3(x)$ does not converge at $x=1$.
\item A modified version of $p_3(x)$ is the following: Let $d>0$.

\noindent
$p_4(x) =  \sum_{n=0}^{\infty} (dx)^n = 1 + dx + (dx)^2 + \cdots$, $r(p_4(x)) = \frac{1}{d}$, but $p_4(\frac{1}{d})$ does not converge at $x=\frac{1}{d}$.
\end{enumerate}
\end{rem}
\begin{rem}
Motivated by Question \ref{quest} for the hyperbolic space, it seems to be natural to consider the following other cases:
\begin{enumerate}
\item $\dim H_*(X;\mathbb Q) = \infty$ and $\dim (\pi_*(X) \otimes \mathbb Q) < \infty$:
In this case we have the homological Hilbert--Poincar\'e series $HP_X(t)$ and the homotopical Poincar\'e polynomial $P^{\pi}_X(t)$ and $P^{\pi}_X(1) < HP_X(1) =\infty$. A real problem would be the following. Let $r_X$ be the radius of convergence of the power series $HP_X(t)$. When $HP_X(r_X)$ does converge, does the following ``Hilali"-type inequality hold?
$$P^{\pi}_X(r_X) \leqq HP_X(r_X).$$ 
\item $\dim H_*(X;\mathbb Q) = \infty$ and $\dim (\pi_*(X) \otimes \mathbb Q) = \infty$:
In this case we have the homological Hilbert--Poincar\'e series $HP_X(t)$ and the homotopical Hilbert--Poincar\'e series  $HP^{\pi}_X(t)$ and $HP^{\pi}_X(1) = HP_X(1) =\infty$. Let $r^H_X$ be the radius of convergence of the power series $HP_X(t)$ and $r^{\pi}_X$ be the radius of convergence of the power series $HP^{\pi}_X(t)$. Let $r_X:= \op{min} \{r^H_X, r^{\pi}_X \}$. When both $HP^{\pi}_X(r_X)$ and $HP_X(r_X)$ do converge (note that if $r^{\pi}_X  < r^H_X$, say, then $HP_X(r^{\pi}_X)$ does converge by the definition of radius of convergence), does the following ``Hilali"-type inequality hold?
$$HP^{\pi}_X(r_X) \leqq HP_X(r_X).$$ \\
\end{enumerate}
\end{rem}
{\bf Acknowledgements:} We would like to thank the referee for his/her thorough reading and useful comments and suggestions. T.Y. is supported by JSPS KAKENHI Grant Number JP20K03591 and S.Y. is supported by JSPS KAKENHI Grant Number JP19K03468. \\

\end{document}